\documentclass[11pt]{article}

\usepackage{tikz}
\usepackage{amsmath}
\usepackage{amssymb}
\usepackage{amsthm}
\usepackage{enumerate}
\newtheorem{thm}{Theorem}
\newtheorem{lemma}{Lemma}[section]
\newtheorem{cor}[lemma]{Corollary}
\newtheorem{definition}[lemma]{Definition}

\makeatletter \@addtoreset{equation}{section} \makeatother

\newcommand{\beq}{\begin{equation}}
\newcommand{\eeq}{\end{equation}}

\def\TS{\textstyle}
\def\ol#1{\overline{#1}}
\def\com#1{\quad{\textrm{#1}}\quad}
\def\eq#1{(\ref{#1})}
\def\nn{\nonumber}

\def\({\left(\begin{array}{cccccc}}
\def\){\end{array}\right)}

\def\bes{\begin{eqnarray}}
\def\ees{\end{eqnarray}}

\def\half{{\TS{\frac12}}}
\def\x#1{$#1\times#1$}

\begin{document}

\title{Shock-free Solutions of the Compressible Euler
  Equations}

\author{Geng Chen\thanks{Department of Mathematics,
Pennsylvania State University, University Park, State College, PA,
16802 ({\tt chen@math.psu.edu}).} \and Robin Young\thanks{Department
of Mathematics and Statistics, University of Massachusetts, Amherst,
MA 01003 ({\tt young@math.umass.edu}).  Supported in part by NSF
Applied Mathematics Grant Number DMS-0908190.}}
\maketitle

\begin{abstract}
We study the structure of shock-free solutions of the compressible
Euler equations with large data.  We describe conditions under which
the Rarefactive/Compressive character of solutions changes, and
conditions under which the vacuum is formed asymptotically.  We
present several new examples of shock-free solutions, which
demonstrate a large variety of behaviors.
\end{abstract}



\section{Introduction}

We consider hyperbolic systems of conservation laws in one space
dimension,
\beq
  u_t+f(u)_x=0, \quad 
  u\in\mathbb{R}^n,\quad 
  f:\mathbb{R}^n\to\mathbb{R}^n.
\label{conservation laws}
\eeq
It is well known that, due to the absence of dissipative effects,
classical ($C^1$) solutions cannot be sustained, and generically,
gradients blow up in finite time.  This is a physical effect which is
manifested by the development of shock waves, at which the conserved
variable $u$ becomes discontinuous.  Once a shock wave forms, one must
study weak solutions, and the analysis becomes much more difficult.

Breakdown of classical solutions for scalar equations is classical,
going back to Bethe and Hopf~\cite{JohnsonCheret}, and was resolved
for \x2 systems by Lax~\cite{lax2,lax3}.  These results state that in
the presence of genuine nonlinearity, nontrivial small data leads to
gradient blowup in finite time.  For larger systems, similar results
are available, again provided the initial data is small~\cite{Fritz
  John,Li daqian,Liu1}.  For \x2 systems, a pair of Riccati-type
equations of the form $\dot w = w^2$ which blow up in finite time can
be derived.  In~\cite{Fritz John}, John derives an analogous system of
equations for the gradient variables, also with quadratic
inhomogeneous part.  These quadratic terms represent interactions
between different nonlinear fields, including self-interaction terms.
For small data, after an initial period of nonlinear wave interaction,
the solution is essentially decoupled into $n$ waves, each propagating
with its own wavespeed.  Each of these waves can be approximately
treated as scalar, and so breaks down in finite time.  When the data
is large, some results are available for particular
systems~\cite{G3,gthesis,G5}, but in general the breakdown of
solutions with large data remains an open problem.  For results in
higher dimensional settings, see~\cite{pansmoller,Rammaha,sideris}.

In this paper, we study the \x3 system of Euler equations of gas
dynamics, which has one linearly degenerate field.  The equations,
representing conservation of mass, momentum and energy, respectively,
are
\begin{align}
  \rho_t +(\rho\,u)_{x'} &= 0, \nn\\
  (\rho\,u)_t + (\rho\,u^2 + p)_{x'} &= 0, \label{euler}\\ \nn
  (\half\,\rho\,u^2 + \rho\,e)_t + (\half\,\rho\,u^3 + u\,p)_{x'} &= 0.
\end{align}
We use a Lagrangian frame, co-moving with the fluid, given by $x =
\int\rho\,dx'$.  The equations become
\begin{align}
  \tau_t-u_x=0,\nn\\
  u_t+p_x=0,\label{lag}\\
  (\half\,u^2+e)_t+(u\,p)_x=0, \nn
\end{align}
and these are equivalent to \eq{euler}~\cite{courant,wagner}.  Here
$\tau=1/\rho$ is the specific volume, $p$ is pressure, $u$ is fluid
velocity, and $e$ is the specific internal energy.  By the Second Law
of Thermodynamics, for classical solutions, the third (energy)
equation can be replaced by the entropy equation,
\beq
  S_t=0,\label{s con}
\eeq
see~\cite{courant}.  The system is closed by specifying a constitutive
law; for convenience, we consider a polytropic ideal $\gamma$-law gas.
For classical solutions, in regions of constant entropy, the first two
equations of \eq{lag} close, forming the $p$-system of
isentropic gas dynamics~\cite{smoller}.

In~\cite{G3}, the first author independently derived a set of Riccati
type equations for the gradients of sound waves, used these to give a
consistent definition of the rarefactive and compressive character
($R/C$ character) of the nonlinear sound waves, and gave conditions
which guarantee shock formation, for data of arbitrary size.  By \eq{s
  con}, the entropy is a linearly degenerate contact field, stationary
for $C^1$ solutions.  These equations are analogous to those of
~\cite{Fritz John,linliuyang}, and can be generalized to other physical
systems~\cite{G5}.  

The presence of a stationary entropy profile of moderate strength
means that nonlinear interactions between fields occur on the same
scale as self-interactions, and can lead to surprising
behavior~\cite{young blake 1}.  In particular, a sound wave can change
its $R/C$ character across a contact discontinuity (entropy jump).
In~\cite{G3,G5}, the entropy field is assumed to be continuous, and
conditions which guarantee gradient blowup are given.  In this paper,
we study the growth of gradients of shock-free solutions with a
varying entropy profile that can include jump discontinuities, and we
consider data having large amplitudes.  In particular, we will
demonstrate the consistency of results of~\cite{G3} for $C^2$ entropy
profiles and~\cite{young blake 1} for contact discontinuities.

For our calculations, it is convenient to consider
{\bf\emph{shock-free solutions}}, in which the velocity $u$ and
pressure $p$ are $C^2$, while the entropy $S(x)$ is $C^2$ except at
finitely many points, and one-sided limits of $S$ and $S_x$ exist
everywhere.  If the initial data are shock-free, the solution remains
shock-free until the first derivatives of $u$ and $p$ blow
up~\cite{majda,Dafermos}.  The following theorem generalizes results
of~\cite{young blake 1,G3}.

\begin{thm}
If the entropy is non-decreasing, a forward $R$ (resp.~$C$) can change
its character only if it crosses backward $C$ waves (resp.~$R$ waves);
a backward $C$ (resp.~$R$) can change only if it crosses a forward $C$
(resp.~$R$).  Symmetric results hold if $m(x)$ is non-increasing,
provided the character of the opposite simple wave is reversed.
\end{thm}

We say that a solution is {\bf\emph{eventually noninteracting}} if all
wave interactions end in finite time: that is, for $t$ large enough,
the solution consists of three regions, defined by outgoing backward,
stationary and forward waves respectively.  The forward and backward
waves are either rarefactions or a single shock; however the profiles
of the rarefactions and contact field need not be explicitly known.  A
Riemann solution is the most obvious example of a solution that is
eventually noninteracting, but we will present several other examples.

We define an {\bf\emph{asymptotic vacuum}} as a characteristic with a
vertical asymptote in the $(x,t)$-plane.  Since the characteristic is
determined by the equation $\frac{dx}{dt} = \pm c$, we necessarily
have $\rho\to 0$ along this characteristic.  Hence, along this
characteristic, the $\rho$, $c$ and $p$ all vanish as $t\to\infty$.
Note that the uniqueness theorem for ODEs implies that the vacuum is
not taken on in finite time unless it is present in the initial data,
see~\cite{Young p system,li1}.

According to \cite{young blake 1}, it is possible to carefully choose an
oscillating entropy profile which supports shock-free, space- and
time-periodic solutions to the Euler equations.  Thus we do not expect
to be able to prove definitive gradient blowup results if the entropy
is non-monotonic, so we largely restrict our attention to monotonic
entropy profiles.

\begin{thm}
Assume that the (variation of the) data is compactly supported and
that the entropy is monotone.  Then a globally defined shock-free
solution is either eventually noninteracting or contains an asymptotic
vacuum.
\end{thm}

In other words, in a monotone entropy field, if a solution continues
to interact for arbitrarily long times and does not contain an
asymptotic vacuum (so in particular \eq{vacuum2} fails), then a shock
necessarily forms in finite time.

In some cases, we can predict the appearance of an asymptotic vacuum
by a condition on the initial data: an asymptotic vacuum can occur in
the solution only if the {\bf\emph{Vacuum Condition}} holds, namely
\beq
   u_0(\infty)-u_0(-\infty) \geq
   m(-\infty)\,z_0(-\infty)+ m(\infty)\,z_0(\infty).
\label{vacuum2}
\eeq
Here $m$ and $z$ are canonical thermodynamic variables which are
nonlinear transformations of the entropy and density, respectively,
given in \eq{cpt}, \eq{taupc} below.  We show that if the
vacuum condition~\eq{vacuum2} holds, then there are no eventually
noninteracting shock-free solutions.  Note that this vacuum condition
is identical to the condition that predicts the existence of the
vacuum in the solution of the Riemann problem~\cite{smoller}, and is
asymptotically equivalent to the condition for an embedded vacuum in
general initial data~\cite{Young p system}.

We then provide several examples of shock-free solutions.  In
particular, if the entropy profile is piecewise constant and
increasing and the data is rarefactive, we show that the boundary of
the interaction region is necessarily a forward characteristic.

We also analyze a case in which the entropy is nonmonotonic.  The
entropy consists of two contacts of equal strength.  By restricting
data to be between the contacts, we show that shock-free solutions are
everywhere rarefactive.  Moreover, in this case we analyze the
long-time behavior of the solution.

\begin{thm}
There are three possible long-time behaviors: asymptotic vacuum;
infinitely reflected waves which converge to the vacuum state; and
infinitely reflected waves with non-vanishing wavespeed and density.
The long-time behavior is determined explicitly by a bifurcation
parameter $\zeta$ determined by the initial data.
\end{thm}

In particular, we obtain interacting solutions which asymptotically
approach vacuum at the specific rate
\[
  s-r = O(1)\,(1+t)^{-(\gamma-1)/(\gamma+1)}, \com{or}
  \tau = O(1)\,(1+t)^{2/(1+\gamma)},
\]
where $s$ and $r$ are the Riemann invariants.  As a consequence, we
note that the presence of an asymptotic vacuum is a stronger condition
than development of vacuum as $t\to\infty$, and in particular,
\eq{vacuum2} can hold even though no asymptotic vacuum is present.

We contrast the cases of monotonic and nonmonotoic entropies: when the
entropy is monotonic, shock-free solutions must be eventually
noninteracting or contain an asymptotic vacuum, while if the entropy
is nonmonotonic, waves can be reflected between the contacts
infinitely often.

Finally, we briefly discuss solutions containing a single shock.  We
treat this as a free boundary problem with specific conditions on
either side of the shock, and discuss global solutions of this type.

\setcounter{thm}{0}
\section{Equations and Wave Curves}

We restrict our attention to a polytropic ideal
gas, with equation of state
\[
  e=c_v T={\frac{p\,\tau}{\gamma-1}} \com{and} p\,\tau=R\,T,
\]
so that
\beq
  p=Ke^{S/c_v}\tau^{-\gamma}.\label{poftau}
\eeq 
Here $S$ is the entropy, $T$ is the temperature, $R$, $K$, $c_v$ are
positive constants, and  $\gamma>1$ is the adiabatic gas constant,
c.f.~\cite{courant}.  The Lagrangian sound speed is given by
\[
	c=\sqrt{-p_\tau}
         =\sqrt{K\gamma}\,{\tau}^{-(\gamma+1)/2}\,e^{S/2c_v}.
\]

We use the coordinates of \cite{young blake 1}: that is, we define
variables $m$ and $z$ so that
\beq
  c = mz^d, \quad p = \int mc\;dz \com{and} \tau = - \int\frac mc\;dz,
\label{cpt}
\eeq
for some $d$.  In fact, it suffices to take $d=\frac{\gamma+1}{\gamma-1}$
and set
\[
	m=C_me^{{S}/{2c_v}} \com{and}
	z=C_z\tau^{-(\gamma-1)/2},
\]
with constants
\[
  C_z = (d-1)^{\frac1{1-d}}  \com{and}
  C_m = \sqrt{K\gamma}C_z^{-d},
\]
and it is easy to check that \eq{cpt} becomes
\beq
	c = m\, z^d, \quad
	p = \frac{m^2\,z^{1+d}}{d+1},  \com{and}
	\tau=\frac{z^{1-d}}{d-1}.
\label{taupc}
\eeq
We shall continue to refer to $m$ as the entropy variable.

In these coordinates, for $C^1$ solutions,
(\ref{lag}) are equivalent to
\begin{align}
	z_t+\frac{c}{m}u_x&=0,\nn\\
	u_t+mcz_x+2\frac{p}{m}m_x&=0,\label{lagrangian zm}\\
	m_t&=0\nn,
\end{align}
the last equation being (\ref{s con}) replacing the energy equation,
valid for smooth solutions.

\subsection{Hugoniot curves}

We describe the shock waves (and contact discontinuties) by the 
Rankine-Hugoniot conditions,
\begin{align}
        \xi \,[\tau] &=-[u], \nn\\
        \xi \;[u] &= [p], \label{RH}\\
        \xi \;[\TS{\frac{1}{2}}\,u^2+e] &= [u\,p], \nn
\end{align}
where $\xi$ is the shock speed, and the brackets denote the change
across the discontinuity, as usual.  Using the identity
$[a\,b] = \ol a\,[b] + [a]\,\ol b$, where $\ol q = (q_0+q_1)/2$, the
third equation of \eq{RH} simplifies to 
\beq
  \xi \;([e] + \ol p\,[\tau]) = 0.
\label{ept}
\eeq

Shock waves correspond to solutions of \eq{RH} with $\xi\ne0$: to
describe these fully, we solve \eq{ept} and use \eq{RH} to determine
$\xi$ and $[u]$.  Using
\[
   e = \frac{p\,\tau}{\gamma-1} \com{and} \gamma = \frac{d+1}{d-1},
\]
equation \eq{ept} becomes
\[
  (d+1)\,\ol p\,[\tau] + (d-1)\,[p]\,\ol\tau = 0.
\]
We now use \eq{taupc} and solve to get
\[
  \left(\frac{m_1}{m_0}\right)^2\;Z^{d+1} 
     = \frac{d\,Z^{d-1}-1}{d-Z^{d-1}}, \com{where}
   Z = \frac{z_1}{z_0},
\]
which we write as
\beq 
	\frac{m_1}{m_0}=f(\frac{z_1}{z_0}),
\label{PWmf}
\eeq 
having defined
\[
  f(Z) :=\sqrt{\frac{Z^{d-1} -\frac{1}{d}}{Z^{d+1}-\frac{1}{d}Z^{2d}}},
   \com{for}
   {d^{\frac{1}{1-d}}}<Z<d^{\frac{1}{d-1}}.
\] 
We now use \eq{RH} to describe $u$ and $\xi$, namely
\[
  [u] = \pm\sqrt{[p]\,[-\tau]} \com{and} \xi = \pm\sqrt{[p]/[-\tau]}.
\]
Substituting and simplifying as above, we obtain
\beq
  u_1-u_0=\pm m_0\,z_0\,g(Z) \com{and}
  \xi =\pm m_0\,z_0^d\,h(Z),
\label{PWgh}
\eeq
where $g$ and $h$ are respectively defined by
\begin{align*}
  g(Z) &:= \frac{1}{\sqrt{d^2-1}}\,
     \sqrt{(f^2(Z)\,Z^{1+d}-1)\,(1-Z^{1-d})}, \\
  h(Z) &:= \sqrt{\frac{d-1}{d+1}}\,
     \sqrt{\frac{f^2(Z)\,Z^{1+d}-1}{1-Z^{1-d}}}.
\end{align*}

Here we have labelled the states on opposite sides of the shock with
the subscripts $0$ and $1$, without explicitly referring to the left
and right states, and these can be interchanged in this description.
We choose the signs in \eq{PWgh} and the value of $Z$ by referring to
Lax's entropy condition: that is, we require that the (absolute)
wavespeed be larger behind the shock, which implies that $Z>1$ if
$z_1$ is behind the shock, and $Z<1$ if $z_1$ is ahead of the shock;
if $z_1$ is the right state, say, we require $Z>1$ for a backward
shock and $Z<1$ for a forward shock.  In this way we can describe both
forward and backward shocks with similar equations, as in~\cite{young
  Global interaction}.  For a further analysis of the structure of the
wave curves and the functions $f$, $g$ and $h$, see~\cite{gthesis}.

\subsection{Stationary Solutions}

The contact discontinuities provide another class of weak solutions,
being solutions of \eq{RH} with wavespeed $\xi=0$.  It is clear that
contacts should satisfy $[u] = [p] = 0$, while the entropy variable
$m$ can vary: in this case, by \eq{taupc} we have
\beq
  m_0^2\,z_0^{1+d} = m_1^2\,z_1^{1+d},
\label{contact}
\eeq
which in turn determines the jump in $z$.  It is well-known that the
entropy is linearly degenerate and these waves are contacts which are
stationary in a Lagrangian frame~\cite{smoller}.  More generally, we
can obtain stationary waves by allowing $m$ (or $z$) to vary while
fixing $u$ and $p$ as constants: this is easily seen by direct
substitution into equations \eq{lag}.  That is, any time-independent
states $(z,u,m)$ given in some region by
\beq
   m = m(x), \quad u = U, \quad  m(x)^2\;z(x)^{1+d} = P,
\label{station}
\eeq
with $U$ and $P$ constants, form a stationary wave solution to \eq{lag}.

\subsection{Isentropic Flow}

A simpler \x2 system, known as isentropic flow, is obtained when the
entropy $S$ (or $m$) is taken to be identically constant, and the
third (energy) equation of \eq{lag} is dropped, to give
\begin{align}
  \tau_t-u_x&=0,\nn\\\label{psys}
  u_t+p_x&=0,
\end{align}
with $p = p(\tau)$, also known as the $p$-system.

The $p$-system is considerably simpler than the full Euler system
because the equations weakly decouple.  Indeed,
for $C^1$ solutions, (\ref{lagrangian zm}) becomes
\begin{align}
  z_t+\frac{c}{m}u_x&=0,\nn\\
  u_t+m\,c\,z_x&=0,\label{p zm}
\end{align}
so that the Riemann invariants, given by
\beq
  r=u-mz\com{and}
  s=u+mz,\label{rs def}
\eeq
respectively, satisfy
\beq
  r^{\backprime}=0 \com{and}
  s^\prime=0.
\label{rs invariant}
\eeq
Here ${}^\backprime$ and $'$ denote differentiation along backward and
forward characteristics, respectively,
\beq
  {}^{\backprime}=\partial_t-c\partial_x  \com{and}
  {}^\prime=\partial_t+c\partial_x.
\label{bkfwd}
\eeq

\section{$R/C$ Character of Solutions}

We briefly recall results from the authors' previous papers
\cite{young blake 1,G3} describing the local rarefactive and
compressive nature of solutions.   

\paragraph{In a constant entropy field}
For isentropic flow~\eq{psys}, the Riemann invariants $r$ and $s$ are
constant along characteristics, and simple waves are described by 
\beq
  m_r=m_l,\quad
  u_r-u_l=m_l(z_a-z_b),
\label{simple}
\eeq
where the subscripts $l,\ r,\ a,\ b$ denote the states to the left,
right, ahead of and behind the wave, respectively.  Noninteracting
simple waves are classified as rarefactive or compressive according to
whether the characteristics diverge or converge, respectively.  For
isentropic flow, this is determined by the profile of the Riemann 
invariants.  Since in the full system, entropy is stationary before
shock formation, the following results extend immediately to
isentropic regions in full \x3 flows:

\begin{definition}
\cite{young blake 1}
\label{def RC}
In a constant entropy field, the local {$R/C$} character of a $C^1$ 
solution is:
\[\left.\begin{array}{lll}
    \text{Forward} & $R$ & \text{iff} \quad s_t< 0,\\
    \text{Forward} & $C$ & \text{iff} \quad s_t> 0,\\
    \text{Backward} & $R$ & \text{iff} \quad r_t> 0,\\
    \text{Backward}&$C$ &\text{iff} \quad r_t< 0.
\end{array}\right.\]
\end{definition}

\begin{lemma}
\label{RCpsystem}
\cite{young Global interaction}
In a constant entropy field, if an interacting solution is $C^2$, the
$R/C$ character of each wave is preserved along characteristics.
\end{lemma}

\begin{lemma}
\label{blowup_psystem}
\cite{lax2,young com} Assume the initial data $r_0(x)$ and $s_0(x)$ of
$r$ and $s$ are $C^2$, and the initial entropy is constant. If $-s_t$
or $r_t$ is negative somewhere in the initial data, then $|u_x|$
and/or $|p_x|$ blow up in finite time.
\end{lemma}

In \cite{lax2}, Lemma~\ref{blowup_psystem} relies on an \emph{a
  priori} assumption that the solution stays away from vacuum; this
assumption is removed in \cite{young com}.  Thus gradients will blow
up (shocks will form) in finite time if and only if there are
compressive waves in the data; see also \cite{li1}.

\paragraph{At a contact discontinuity\label{subsection_contact}}

Simple waves preserve their character in isentropic regions, as the
(derivatives of) Riemann invariants propagate with the wave.  However,
waves may change type when crossing a contact discontinuity which
separates different isentropic regions.  According to \eq{contact}, 
the change in variables is
\beq
  u_r=u_l \com{and}
  m_r\,z_r = m_l\,z_l\,Q,
\label{z_q_rel}
\eeq
where we have set
\beq
  Q = \big(\frac{m_r}{m_l}\big)^{\frac{d-1}{d+1}},
\com{so also}
      \frac{z_r}{z_l}= Q^{\frac{-2}{d-1}}.
\label{q def}
\eeq
with corresponding changes in the derivatives of Riemann invariants
$r_t$ and $s_t$ by \eq{rs def}.  It follows that if forward and
backward simple waves cross the jump simultaneously, then one of the
waves could change character.  Following \cite{gthesis}, we will call
the contact discontinuity a $1$-contact if the entropy decreases,
$m_l>m_r$ (so $Q<1$), and a $3$-contact if $m_l<m_r$ ($Q>1$).

\begin{lemma}
\label{Then RC Young}
\cite{young blake 1}
A nonlinear wave changes its {$R/C$} value at a contact discontinuity
when and only when one of the following inequalities hold:
\[\left.\begin{array}{lll}
   R_{in}^- \rightarrow C_{out}^- &
   \text{iff} & Q m_l \dot{z}_l<\dot{u}_l< m_l \dot{z}_l,\\
   C_{in}^- \rightarrow R_{out}^- &
   \text{iff} &  m_l \dot{z}_l<\dot{u}_l< Q m_l \dot{z}_l,\\
   R_{in}^+ \rightarrow C_{out}^+ &
   \text{iff} & -Q m_l \dot{z}_l<\dot{u}_l< -m_l \dot{z}_l,\\
   C_{in}^+ \rightarrow R_{out}^+ & \text{iff} &  -m_l
   \dot{z}_l<\dot{u}_l< -Q m_l \dot{z}_l,
\end{array}\right.\] 
where $\dot y:=y_t$ denotes the time derivative, the subscripts denote
incoming and outgoing waves (or the side of the jump), and the
superscripts indicate the direction of the wave: $-$ is backward, $+$
is forward.
\end{lemma}

Note that the conditions of the lemma are mutually exclusive, so only
one wave can change its character at any time.  Moreover, for a fixed
jump, a change in one wave is possible only if the opposite wave has
the right character.

\begin{cor}
\label{TYcor}
At a $3$-contact ($Q>1$), the backward wave can change from $R$ to
$C$ (resp.~$C$ to $R$) only if both the incoming and outgoing forward
waves are $R$ (resp.~$C$); the forward wave can change from $C$ to $R$
(resp.~$R$ to $C$) only if both backward waves are $R$ (resp.~$C$).
Similar conclusions hold for a $1$-contact, but the character of the
incoming opposite wave changes.
\end{cor}

For later use, we record the change of Riemann invariants across the
jump: it follows easily from \eq{rs def}, \eq{z_q_rel} that
\beq
     r_r = \frac{1+Q}2\,r_l+\frac{1-Q}2\,s_l,
\com{and}
     s_r = \frac{1-Q}2\,r_l+\frac{1+Q}2\,s_l.
\label{rsjump}
\eeq

\paragraph{For non-isentropic smooth solutions\label{subsection_shockfree}}

In \cite{G3}, the first author provides an appropriate definition of
the $R/C$ character for the full (non-isentropic) Euler equations.
Recalling \eq{bkfwd}, define the quantities
\begin{align}
 \alpha &:= -{p^{\backprime}}/{c^2}
   =u_x+mz_x+\TS{\frac{2}{d+1}}\,m_x\, z\com{and}\nn\\
\beta &:= -{p^{\prime}}/{c^2}
   =u_x-mz_x-\TS{\frac{2}{d+1}}\,m_x\, z;
\label{def_beta}
\end{align}
these are multiples of derivatives of Riemann invariants.

\begin{definition}
\label{def2}
\cite{G3,G5}
The local {$R/C$} character in a $C^1$ solution is
\[\begin{array}{llll} 
	\text{Forward}& $R$ &\text{iff}& \alpha>0,\\ 
	\text{Forward}& $C$ & \text{iff} & \alpha<0,\\
	\text{Backward}& $R$ & \text{iff}& \beta>0,\\ 
	\text{Backward}& $C$ &\text{iff}&\beta<0.
\end{array}\]
\end{definition}

For $C^1$ solutions, it is easy to show using \eq{taupc},
(\ref{lagrangian zm}) and \eq{rs def}, that
\[
	s_t+c\,\alpha=0,\quad r_t-c\,\beta=0,
\]
so this definition agrees with and extends Definition \ref{def RC}.

The following Riccati type equations describe the growth of gradients:
\begin{lemma}
\label{remark}
\cite{G3}
If the solution of (\ref{lag}) is $C^2$, then
\begin{align}
  \alpha^\prime&=
     k_1\big( k_2(3\alpha+\beta)+\alpha\beta-\alpha^2\big) \com{and}\nn\\
	\beta^\backprime&=k_1{\big(}
        -k_2(\alpha+3\beta)+\alpha\beta-\beta^2\big),\label{rem12} 
\end{align}
where 
\beq
  k_1 :=\frac{\gamma+1}{2(\gamma-1)} z^{\frac{2}{\gamma-1}} \com{and}
  k_2 :=\frac{\gamma-1}{\gamma(\gamma+1)}z m_x.
\label{k def}
\eeq
Moreover,
\beq
  |\alpha|\ \text{or}\ |\beta|\to
  	\infty \com{iff} |u_x|\ {\it or}\ |p_x| \to \infty.
\label{alpha beta blowup}
\eeq
\end{lemma}

We note that these equations are similar to those derived by F.~John
in \cite{Fritz John}, but were independently derived from a different
point of view by the first author in \cite{G3}.  Condition \eq{alpha
  beta blowup} coincides exactly with formation of a shock wave.

\subsection{Global $R/C$ Structure}

In a constant entropy field, the $R/C$ character of waves is preserved, but
in a varying entropy field, it may change.  We describe conditions
under which the $R/C$ character of an interacting wave changes.
Essentially, the only way a wave can change is if it is nonlinearly
superposed with reflections from the interaction of \emph{opposite}
waves with the background entropy field; this is consistent with the
changes across a contact described above.

\begin{thm}
\label{global RC}
Suppose a solution satisfies the Shock-free Condition.  If the
entropy $m(x)$ (i.e.~$S(x)$) is non-decreasing, a forward $R$
(resp.~$C$) can change its character only if it crosses backward $C$
waves (resp.~$R$ waves); a backward $C$ (resp.~$R$) can change only if
it crosses a forward $C$ (resp.~$R$).  Symmetric results hold if
$m(x)$ is non-increasing, provided the character of the opposite
simple wave is reversed.
\end{thm}

It follows from the proof that this statement includes changes of type
from zero strength waves to $C$ or $R$.

\begin{proof}
First suppose the entropy is $C^2$, and consider a forward
rarefaction.  We consider the evolution of $\alpha$ along the forward
characteristic, propagating through a field of non-decreasing
entropy.  Also, suppose that $\beta\geq0$ along this characteristic,
so that our forward wave crosses no backward compressions.  Let
$\Gamma$ denote the forward characteristic, parameterized by $t_0\leq t$.

We prove by contradiction that $\alpha(t)>0$ on $\Gamma$.  Suppose
not, and let $t_*$ be the first time for which $\alpha(t)=0$ along
$\Gamma$.  Since $\beta\geq 0$ along $\Gamma$, $k_1$ and $k_2$ are
non-negative, and $\alpha(t)>0$ for $t_0<t<t_*$, by \eq{rem12} we have
\bes
   \alpha^\prime\geq 3 k_1 k_2\alpha-k_1\alpha^2
   \com{for} t_0<t<t_*.
\label{alphaineq}
\ees
Denote 
\[
  {\tilde{\alpha}}(t) :=\alpha(t) e^{-\int_{t_0}^{t}3 k_1 k_2 dt}
\com{and}
  k_+ :=k_1 e^{\int_{t_0}^{t}3 k_1 k_2 dt},
\]
where the integral is along $\Gamma$.  Using the integrating factor
$e^{-\int_{t_0}^{t}3 k_1 k_2 dt}$, \eq{alphaineq} yields
\[
  {\tilde{\alpha}}^{\prime}\geq-k_{+}{\tilde{\alpha}^2}.
\]
Dividing by $\tilde{\alpha}^2$ and integrating along $\Gamma$, we get
\[
  \frac{1}{\tilde{\alpha}(t)}\leq \int_{t_0}^{t}k_{+}
  dt+\frac{1}{\tilde{\alpha}(t_0)}.
\]
Since $\tilde\alpha(t_0)>0$ and $\tilde\alpha(t_*)=0$, we must have
\[
  \lim_{t\to {t_*}^-}{\int_{t_0}^{t}k_{+}
      dt}=+\infty,
\]
which contradicts the Shock-free condition.   We conclude that
$\alpha>0$ on $\Gamma$, with the lower bound
\[
  {{\alpha}(t)}\geq \frac{e^{\int_{t_0}^{t}3k_1 k_2
      dt}{\alpha}(t_0)}{1+{\alpha}(t_0)\int_{t_0}^{t}k_{+}
    dt}.
\]

Now consider a point at which the entropy is not $C^2$, (actually a
contact, since entropy is stationary).  By the Shock-free condition,
one-sided limits of $m$ and $m_x$ exist and $u$ and $p$ are $C^2$, so
our $R/C$ variables $\alpha$ and $\beta$ and characteristic $\Gamma$
are defined up to $x=x_*$ with well-defined one-sided limits.  If
$m(x)$ is continuous at $x_*$, then $\alpha=-\frac{s_t}{c}$ and
$\beta=\frac{r_t}{c}$ are also continuous, so the above argument
yields $\alpha(t_*+)>0$, and we continue the characteristic forward in
time.  If the entropy has a jump at $x_*$, then Corollary~\ref{Then RC
  Young} applies, and the conclusion of the theorem follows.

The proofs of other cases are entirely similar, and omitted.
\end{proof}

\paragraph{Boundary $R/C$ character structure}

Finally we consider the $R/C$ structure at the edge of the support of
the entropy profile.  The main case of interest is that of no incoming
wave and an outgoing rarefaction, corresponding to initial waves being
compactly supported and no shocks outside the support of the entropy,
respectively.  

\begin{lemma}
\label{BoundaryRC}
Suppose that the solution satisfies the Shock-free condition.  A wave
emerging from a region of varying entropy keeps its $R/C$
character as long as there are no incoming waves of the other family.
If $m(x)$ is non-decreasing, a forward $R$ (resp.~$C$) emerging to the
right reflects a backward $C$ (resp.~$R$) back into the region of
varying entropy.  A backward $R$ (resp.~$C$) emerging to the left of
the varying entropy reflects a forward $R$ (resp.~$C$) into the region
of varying entropy.  Similar results hold if the entropy $m(x)$ is
non-increasing, but the character of the reflected wave is reversed.
\end{lemma}

\begin{proof}
We consider only the first case: a forward rarefaction emerging from
the right edge $x_1$ of the varying entropy, which is non-decreasing.
All other cases are similar.

First suppose there is a jump at $x_1$: by Def.~\ref{def RC} and
\eq{rs def}, the states to the right of the jump satisfy
\[
  \dot r_r = \dot u_r - m_r\,\dot z_r = 0 \com{and}
  \dot s_r = \dot u_r + m_r\,\dot z_r < 0,
\]
so that $\dot z_r<0$, where $\dot y:=y_t$ denotes the time
derivative.  Applying \eq{z_q_rel} and using $Q>1$, we see that
\begin{align*}
  \dot r_l &=\dot u_l - m_l\,\dot z_l = m_r\,\dot z_r\,(1-1/Q) <0
\com{and}\\ 
  \dot s_l &= m_r\,\dot z_r\,(1+1/Q) <0,
\end{align*}
so that the waves on the left are a forward rarefaction and backward
compression.

Now suppose the entropy is smooth ($C^1$) and increasing up to $x_1$,
but constant for $x\ge x_1$.  Then also, for $x\ge x_1$, we have
\[
  \alpha > 0 \com{and} \beta = 0,
\]
since there are no incoming backward waves and the emerging forward
wave is a rarefaction.  By continuity, for $x$ near $x_1$ and $x<x_1$
we have $\alpha>0$, $\beta\approx 0$, and, by \eq{k def}, $k_2>0$.
Thus the forward wave is a rarefaction for $x$ near $x_1$, and
moreover, in this neighborhood,
\[
  \beta^\backprime < 0, \com{so also} \beta <0,
\]
so the backward characteristic reflected back into
the varying entropy region is compressive.
\end{proof}

\section{Shock-free solutions with a single contact}

Our main results refer to interacting solutions which interact only
for finite times or contain asymptotic vacuums.  When there is a
single entropy jump, say at $x=0$, then these are the only
possibilities for shock-free solutions.  By first treating a single
contact discontinuity, we avoid issues of multiple reflections of
waves considered in later sections.

Consider the interaction of smooth isentropic waves at a contact 
discontinuity separating constant entropy values $m_l$ and
$m_r$.  We describe the states on either side of the contact by 
\[
  z_l(t) = z(0-,t), \quad z_r(t) = z(0+,t) \com{and} 
  u_l(t) = u_r(t) = u(0,t).
\]
We find the $R/C$ character of the incoming and outgoing waves by
using (\ref{rs def}) to calculate the Riemann invariants, and
differentiating these in time.

\begin{lemma}
\label{lemma single contact}
Suppose the interacting solution contains no shocks.  The solution is
eventually noninteracting if and only if the Vacuum Condition
\eq{vacuum2} is not satisfied.  If the Vacuum Condition holds, the
solution contains an asymptotic vacuum.
\end{lemma}

\begin{proof}
Suppose that interaction ends in finite time $T\gg1$, and use
subscripts $g$ and $d$ to denote the constant states on the left and
right of the contact after the interaction has completed,
respectively, so that $u_g = u_l(T)$, etc.  Since there are
no shocks, the outgoing waves leaving the interaction region must be
rarefactions.  Using \eq{simple}, we relate these states to the extreme states
(subscripted by $\pm\infty$) as
\begin{align}
  u_g - u_{-\infty} &= m_{-\infty}\,(z_{-\infty} - z_g), \nn\\
  u_\infty - u_d &= m_{\infty}\,(z_{\infty} - z_d).
\label{ugd}
\end{align}
We now use \eq{z_q_rel} to relate the states across the jump,
\[
  u_d = u_g \com{and}  m_d\,z_d = m_g\,z_g\,Q.
\]
Now since $m_g=m_{-\infty}$ and $m_d=m_{\infty}$, we get
\[
  u_\infty - u_{-\infty} - m_{-\infty}\,z_{-\infty} 
    - m_\infty\,z_\infty = - (1+Q)\,m_{-\infty}\,z_g < 0,
\]
so that \eq{vacuum2} fails.

Now suppose that the vacuum condition holds.  Then the interaction
persists for all time, and because there are no shocks, the outgoing
waves must both be rarefactions.  For $t>0$, trace forward and
backward characteristics back from the contact at $x=0$ to the initial
time $t=0$, to define functions $x_-(t)<0$ and $x_+(t)>0$,
respectively.  Thus the forward rarefaction starting at $x_-(t)$ meets
the contact at $t$, etc., and since characteristics cannot intersect,
we have $\dot x_-\le0$ and $\dot x_+\ge0$.  Now, since Riemann
invariants are preserved on characteristics, we have
\begin{align*}
  u_l(t) + m_l\,z_l(t) &= u_0(x_-(t)) + m_l\,z_0(x_-(t)) \com{and}\\
  u_r(t) - m_r\,z_r(t) &= u_0(x_+(t)) - m_r\,z_0(x_+(t)),
\end{align*}
where $(z_0(x),u_0(x))$ is the initial data.  Using \eq{z_q_rel}, we
conclude that
\begin{align*}
  u_0(x_-&(t)) - u_0(x_+(t)) + m_l\,z_0(x_-(t)) + m_r\,z_0(x_+(t)) 
  \\ &= (1+Q)\,m_l\,z_l(t) > 0.
\end{align*}
Now since \eq{vacuum2} holds, at least one of $x_-(t)$ or $x_+(t)$
must converge to some finite $x_*$ as $t\to\infty$.  It follows that
the incoming characteristic beginning at $x_*$ (and also those
starting further out) cannot meet the contact in finite time, so
remains bounded for all time.  Thus this characteristic has a vertical
asymptote, and the solution contains an asymptotic vacuum.
\end{proof}

\begin{cor}
\label{vcf}
If a shock-free solution is eventually noninteracting, then the vacuum
condition \eq{vacuum2} fails, whatever the entropy profile.
\end{cor}

\begin{proof}
The proof proceeds as above and \eq{ugd} continues to hold, provided
subscripts $g$ and $d$ refer to the states on either side of the
varying entropy.  For large times $t>T$, the solution restricted to
the entropy profile is a stationary entropy wave, across which
velocity $u$ and pressure $p$ are constant.  Thus $u_g=u_d$ and
\eq{ugd} yields
\[
  u_\infty - u_{-\infty} - m_{-\infty}\,z_{-\infty} 
    - m_\infty\,z_\infty = - m_{-\infty}\,z_d - m_\infty\,z_g< 0,
\]
so the vacuum condition fails.
\end{proof}

In isentropic flow, the asymptotic vacuum is produced only by the
interaction of two opposite rarefaction waves~\cite{lizhao,Young p
  system}.  When a contact is present, a single strong rarefaction can
produce an asymptotic vacuum, since the leading edge of the
rarefaction crossing the contact may produce a reflected rarefaction,
and the interaction of the initial strong rarefaction with the
reflected rarefaction can lead to the vacuum.  According to
Corollary~\ref{TYcor}, one of the incoming waves must be a pure
rarefaction, while the other can contain compressive regions which are
changed by the interaction at the contact, or there could be no
opposite incoming wave.

To be specific, consider a backward rarefaction (initially compactly
supported in $(0,\infty)$) interacting with a $3$-contact (increasing
entropy jump), $m_r > m_l$, at $x=0$, with no incoming forward wave.
The profile of the outgoing waves is determined by the traces of the
states on either side of the contact discontinuity.  By \eq{z_q_rel}
and \eq{simple}, the initial data consists of constant states
$(z_l,u_l,m_l)$ and $(z_m,u_l,m_r)$ satisfying
\[
  m_r\,z_m = m_l\,z_l\,Q,
\]
together with a backward rarefaction given by
\[
  u_0(x) - u_l = m_r\,z_m - m_r\,z_0(x),
\]
for a decreasing function $z_0(x)$ with $z_0(0) = z_m$.  If there is
some $x_*$ such that
\[
  u_0(x_*) - m_r\,z_0(x_*) = u_l + m_l\,z_l,
\]
that is,
\[
  m_r\,z_0(x_*) = \frac{Q-1}2\,m_l\,z_l, \com{or}
  z_0(x_*) = \frac{Q-1}{2\,Q}\,z_m,
\]
then all backward characteristics beginning to the left of $x_*$ cross
the contact, while all those starting at $x\ge x_*$ asymptote, with
corresponding states approaching vacuum.  Note that the support of the
vacuum in the limit $t\to\infty$ is some interval of the form
$[0,x_\#]$.

\section{Shock-free solutions with Monotone Entropy}

We now prove our main theorem, which describes the structure of
shock-free solutions with a monotone entropy profile.

\begin{thm}
\label{mainthm}
Assume that the (variation of the) data is compactly supported and
that the entropy is monotone.  Then a globally defined shock-free
solution is either eventually noninteracting or contains an asymptotic
vacuum.
\end{thm}

\begin{proof}
Assume without loss of generality that the entropy is non-decreasing.
We assume that the solution is shock-free, has no asymptotic vacuum,
and is not eventually noninteracting, and derive a contradiction.
Denote the interval on which the entropy varies by $[x_0,x_1]$.
Since the data is compactly supported, there is some $T\gg1$
such that
\[
  \beta(x_1+,t) = 0 \com{and}
  \alpha(x_0-,t)=0  \com{for every} t\ge T.
\]
Next, since there are no shocks, by Lemmas~\ref{RCpsystem}
and~\ref{blowup_psystem}, the outgoing waves must be rarefactions, so
that
\[
  \alpha(x_1+,t) \ge 0 \com{and}
  \beta(x_0-,t) \ge 0 \com{for each} t \ge 0.
\]
Because there is no asymptotic vacuum, all forward and backward
characteristics pass through the interval $[x_0,x_1]$ of varying
entropy in finite time.

It follows from Lemma~\ref{BoundaryRC} that
\[
  \alpha(x_1-,t) \ge 0 \com{and}
  \beta(x_1-,t) \le 0 \com{for} t \ge T.
\]

For each $t\ge T$, denote the backward characteristic starting from
the point $(x_1,t)$ by $\Gamma_-(t)$, and define
\[
  x_*(t) = \min\{x_0\le x \le x_1\,|\,
        \beta(\xi+,\tau)\le 0,\ \forall \xi \ge x,\ 
        (\xi,\tau) \in \Gamma_-(t)\},
\]
so $x_*(t)$ is the first possible point on the characteristic that
$\beta$ becomes positive.  The curve $x_*(t)$ is continuous, since
$\beta$ is continuous away from contacts, and if $\beta$ first changes
sign at a contact, $\beta(x_c+)\le 0\le \beta(x_c-)$, then we have
$x_*(t)=x_c$.  In this case, continuity of $x_*(t)$ follows since
$\beta(x_c-)$ is a continuous function of $\alpha(x_c+)$ and
$\beta(x_c+)$.

By definition, we have $\beta \le 0$ on the right of the curve
$x_*(t)$, while also $\beta=0$ on the curve (except possibly at a
contact).  Away from a contact, we must have $\beta^\backprime\ge0$, so that 
$\alpha\le 0$ at $x_*(t)$, by \eq{rem12}.  At a contact,
Corollary~\ref{TYcor} also implies that the forward wave is $C$, so again
$\alpha\le0$. 

Now, since the solution is not eventually noninteracting, there is
some $t_\#$ such that
\[
  \alpha(x_1-,t_\#) > 0, \com{and} \beta(x_1-,t_\#) < 0,
\]
by Lemma~\ref{BoundaryRC}.  We now trace the forward characteristic
$\Gamma_+$ back from $(x_1-,t_\#)$ until it first meets the curve
$x_*(t)$ at the point $(x_\dag,t_\dag)$, say.  We now have
\[
  \alpha(x_\dag,t_\dag) \le 0 \com{and} \alpha(x_1-,t_\#) > 0,
\]
while also 
\[
  \beta(x,t) \le 0 \com{for} x_\dag\le x\le x_\#,\ (x,t)\in\Gamma_+,
\]
which together contradict Theorem~\ref{global RC}.
\end{proof}

\subsection{Piecewise Constant Entropy}

We can describe the structure of solutions in more detail if we make
the further simplifying assumption that the entropy is piecewise
constant and monotone non-decreasing.  By this we mean that the
entropy has finitely many jumps, while $u$ and $p$ remain $C^2$.  Our
results apply directly to monotone non-increasing piecewise constant
entropy with appropriate modification, and we expect that the
extension to general shock-free solutions with varying monotone
entropy holds with technical changes in the proofs.

\begin{lemma}
\label{noSF}
Suppose the entropy is piecewise constant non-increasing.  If the
initial data are never forward compressive but somewhere backward
compressive, then there are no shock-free solutions.
\end{lemma}

\begin{figure}[htp]
\centering
\begin{tikzpicture}[scale=0.8]
\draw[->] (-3,0) -- (-3,2);
\node at (-3.2,1.8) {$t$};
\draw[->] (-3,0) -- (-1,0);
\node at (-1.2,-0.3) {$x$};
\foreach \xx in {1,1.4,3,4,5}
  \draw[dashed] (\xx,0) -- (\xx,4);
\draw (2.6,0) .. controls (2,0.5) and (1.6,0.8) .. (1.5,4);
\draw (1.6,3.3) -- (1.76,2.7);
\foreach \dd in {0.1,0.2,0.3}
{
  \draw (1.8,2+\dd) -- (2+.5*\dd,1.35+1.5*\dd);
  \draw (2.4,1+\dd) -- (2.8+.3*\dd,0.5+1.5*\dd);
  \draw (3.1,0.4+\dd) -- (3.6+.4*\dd,0.0+1.5*\dd);
  \draw (4.1,1.3+\dd) -- (4.7+.5*\dd,0.8+1.6*\dd);
  \draw (3.3-0.8*\dd,0.7+0.9*\dd) -- (3.9-0.05*\dd,1.3+0.5*\dd);
  \draw (5.8+4*\dd,0) -- (5.1,0.8+\dd);
  \draw (3.1,2.2+2*\dd) -- (3.9-0.05*\dd,1.5+0.8*\dd);
}
\node at (5,-0.3) {$x_1$};
\node at (1,-0.3) {$x_0$};
\node at (2.2,3.8) {$\mathcal B$};
\end{tikzpicture}
  \caption{Proof of Lemma~\ref{noSF}}
\end{figure}
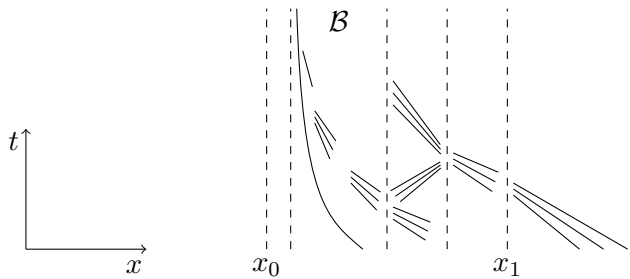

\begin{proof}
If a backward $C$ leaves the varying entropy field, then shocks form
in finite time by Lemma~\ref{blowup_psystem}.  If not, because there
are only finitely many isentropic blocks, we can isolate the left-most
isentropic block $\mathcal B$ containing some backward $C$.  By
Corollary~\ref{TYcor}, forward $C$ can form only at a $3$-contact when
the crossing backward waves are $C$.  Thus there is no forward $C$ in
block $\mathcal B$ or the blocks to the left of it.  Hence the
backward $C$ in this block cannot be cancelled.  Thus the backward $C$
must be obstructed by a backward asymptotic vacuum in (isentropic)
block $\mathcal B$.  But in this case, a singularity forms in finite
time by Lemma~\ref{blowup_psystem}.
\end{proof}

\begin{lemma}
\label{pcni}
Suppose the entropy is piecewise constant non-increasing.  If the
initial data are never compressive, then shock-free solutions are
never compressive.  For such solutions, we also have:
\begin{enumerate}[(a)]
\item If the data is constant to the right of the entropy jumps, and
  if there is some initial forward $R$, then there are no shock-free
  solutions.
\item If the solution is shock-free and eventually noninteracting,
  then the upper boundary of the interaction region is a
  \emph{forward} characteristic which does not end at an interior
  contact.
\end{enumerate}
\end{lemma}

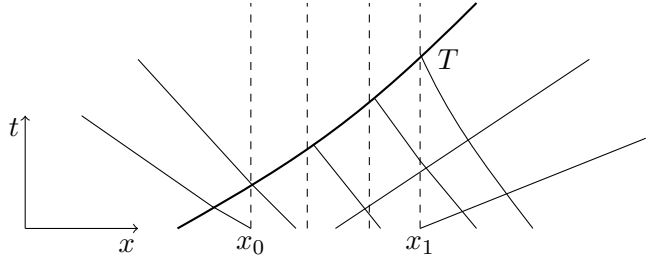
\begin{figure}[htp] \centering
\begin{tikzpicture}[scale=0.75]
\draw[->] (-3,0) -- (-3,2);
\node at (-3.2,1.8) {$t$};
\draw[->] (-3,0) -- (-1,0);
\node at (-1.2,-0.3) {$x$};
\foreach \xx in {1,2,3.1,4}
  \draw[dashed] (\xx,0) -- (\xx,4);
\draw[thick] (-0.3,0) .. controls (2,1.3) and (3,2) .. (5,4);
\draw (6,0) .. controls (5,1.3) and (4.5,2) .. (4,3.1);
\draw (1.8,0) -- (1,0.8) -- (-1,3);
\draw (1,0) -- (0.3,0.4) -- (-2,2);
\draw (5,0) -- (4,1.2) -- (3.2,2.3);
\draw (3.3,0) -- (2.1,1.5);
\draw (2.5,0) -- (7,3);
\draw (4,0) -- (8,1.6);
\node at (4.5,3) {$T$};
\node at (1,-0.3) {$x_0$};
\node at (4,-0.3) {$x_1$};
\end{tikzpicture}
  \caption{Shock-free eventually noninteracting solution with
    rarefaction data}
\end{figure}

\begin{proof}
Recall from Corollary~\ref{TYcor} that a forward wave can change from
$R$ to $C$ at a $3$-contact, only when the crossing backward waves are
$C$.  Thus, any backward $C$ must appear earlier than forward $C$.
Since there are no $C$ in the initial data, by Lemma~\ref{noSF},
shock-free solutions are always non-compressive.  This proves the
first statement.

We now note the following consequences of Corollary~\ref{TYcor}:
\begin{enumerate}
\item
If there are no incoming backward waves while the incoming forward
waves are rarefactive on some piece of the right-most $3$-contact,
then the reflected backward waves are compressive.  If this happens, a
shock necessarily forms in finite time, by above.
\item
In a shock-free solution, a forward $R$ cannot be cancelled since the
backward waves are nowhere $C$.  
\end{enumerate}

Statement (a) now follows from 1 and 2, because the initial forward
$R$ remains $R$ until it meets the last contact, and so reflects a
backward $C$.

Finally assume the solution is shock-free and eventually
noninteracting.  Let $T$ denote the maximum time at which waves cross
the right-most entropy jump (at $x_1$).  We claim the forward
characteristic traced back from $(x_1,T)$ is the upper bound of the
interaction region.  Since our shock-free solution is nowhere $C$, any
backward $R$ above this characteristic would reflect a forward $R$,
which in turn reflects a backward $C$ at $x_1$.  If this
characteristic ended at an interior contact, a backward rarefaction
would emerge, and later reflect another forward wave, a contradiction.
\end{proof}

\begin{lemma}
\label{AV}
Suppose the entropy is piecewise constant non-increasing.  If the
initial data are nowhere compressive, then shock-free solutions
contain an asymptotic vacuum if and only if \eq{vacuum2} holds.
\end{lemma}

\begin{proof}
By Corollary~\ref{vcf} and Theorem~\ref{mainthm}, we need only show
that if there is an asymptotic vacuum, then \eq{vacuum2} holds.

First, we claim that the right-most backward characteristic, denoted
by $x=\Phi(t)$ forms an asymptotic vacuum to the right of the entropy
field, so that $\Phi(\infty)\ge x_1$.  If not, the interaction must
end in finite time by Lemma~\ref{pcni}(b).  Next, for any time $T$,
there exists a forward characteristic $\Psi$ which crosses all
$3$-contacts and which meets $x_1$ at some $t_*>T$, so that
\[
  \Psi(0) < x_0 \com{and} \Psi(t_*) = x_1, \quad t_*>T.
\]
Moreover, by choosing $T$ large enough, we can assume $\Psi(0)$ is to
the left of the support of the initial data.  We denote by
$\psi_0,\dots,\psi_n$ the points at which $\Psi$ crosses the $x$-axis
and contacts, respectively, as in Figure~\ref{vacpf}.  Also set
$\phi_0=(\Phi(0),0)$ and $\phi_1=(\Phi(t_*),t_*)$.

\begin{figure}[htp] \centering
\begin{tikzpicture}[scale=0.65]
\draw[->] (-3.5,0) -- (-3.5,2);
\node at (-3.7,1.8) {$t$};
\draw[->] (-3.5,0) -- (-1.5,0);
\node at (-1.7,-0.3) {$x$};
\foreach \xx in {1,2,3.1,4,4.8,6}
  \draw[dashed] (\xx,0) -- (\xx,7);
\draw[thick] (-0.3,0) .. controls (2,1.3) and (3,2) .. (6,6);
\draw[thick] (8,0) .. controls (7,1.3) and (6.5,2) .. (6.2,7);
\draw[dotted] (0,6) -- (8,6);
\node at (-0.5,6) {$t_*$};
\node at (1,-0.3) {$x_0$};
\node at (6,-0.3) {$x_1$};
\node at (8.3,0.3) {$\phi_0$};
\node at (6.6,6.3) {$\phi_1$};
\node at (-0.5,0.3) {$\psi_0$};
\node at (0.7,1) {$\psi_1$};
\node at (2.8,2.7) {$\psi_k$};
\node at (5.7,6.3) {$\psi_n$};
\node at (7,3) {$\Phi$};
\node at (5.3,4.6) {$\Psi$};
\end{tikzpicture}
  \caption{Asymptotic vacuum}\label{vacpf}
\end{figure}
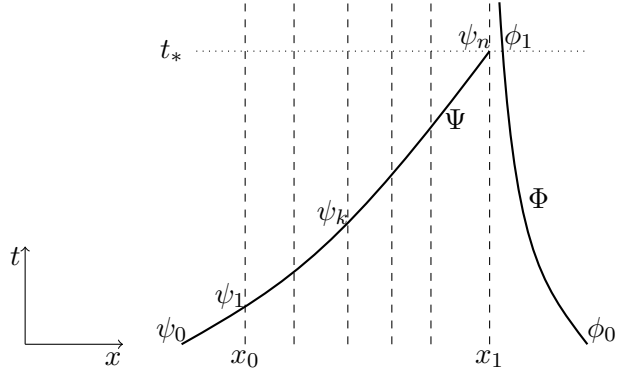

Since there are no compressions, we have for any $(x,t)$,
\[
  u_x = \frac{r_x+s_x}2 \ge 0,
\]
and recall that $u$ is $C^2$.  We write
\begin{align}
  u_0(\infty) - u_0(-\infty) &= u(\phi_0)-u(\phi_1)
  	+u(\phi_1)-u(\psi_n)+u(\psi_n)-u(\psi_0) \nn\\
  &\ge  u(\phi_0)-u(\phi_1)+u(\psi_n)-u(\psi_0).
\label{ueq}
\end{align}
Now we use \eq{simple} to write
\[
  u(\phi_0)-u(\phi_1) = m_\infty\,z_0(\infty) - m_\infty\,z(\phi_1),
\]
and, telescoping, we write
\begin{align*}
  u(\psi_n)-u(\psi_0) &= u(\psi_n)-u(\psi_{n-1}+)
         +\dots +u(\psi_1-)-u(\psi_0)\\
  &= m_n\,z(\psi_{n-1}+)-m_n\,z(\psi_n-) +\dots \\
    &\qquad\qquad{}+ m_{-\infty}\,z(\psi_0)-m_{-\infty}\,z(\psi_1-)\\
  &>{} -m_n\,z(\psi_n-)+ m_{-\infty}\,z_0(-\infty),
\end{align*}
where we have used
\[
  m_{k+1}\,z(\psi_k+) = Q_k\,m_k\,z(\psi_k-) > m_k\,z(\psi_k-),
\]
by \eq{z_q_rel}.  Here the $m_k$ are the intermediate entropy levels
(with $m_0=m_{-\infty}$), and $Q_k>1$ the corresponding entropy jumps.
Equation \eq{ueq} thus becomes
\begin{align*}
  u_0(\infty) - u_0(-\infty)\ge 
     m_\infty&\,z_0(\infty)+ m_{-\infty}\,z_0(-\infty)\\
    & {}  - m_\infty\,z(\phi_1)-m_n\,z(\psi_n-),
\end{align*}
and the last two terms vanish in the limit as $t_*\to\infty$, yielding
\eq{vacuum2}.
\end{proof}

\section{Examples of shock-free solutions}

When the entropy is monotonic, a nontrivial shock-free solution must
contain an asymptotic vacuum or must be eventually non-interacting.
Here we describe a general method for constructing eventually
noninteracting shock-free solutions.  

Begin by specifying a piecewise smooth ($C^2$) compactly supported
entropy profile.  Fix some constant $x_0$ (say the location of an
entropy jump), and specify data $(z,u)$ or $(r,s)$ on a compact
$t$-interval at this $x_0$.  Treat this as Cauchy data and evolve it
\emph{spatially} in both forward and backward directions.  The
equations for spatial evolution form a \x2 system with varying
coefficients, while the entropy is smooth.  At an entropy jump, the
jump is resolved by the Hugoniot conditions \eq{contact}, which for a
$\gamma$-law gas is simply a \x2 linear map \eq{z_q_rel}.

In order to obtain global existence, we require only an \emph{a
  priori} $C^1$ estimate~\cite{li1,liyu,majda}.  That is, we obtain a
shock-free solution as long as $\alpha$ and $\beta$ remain bounded in
the half-plane $t>0$.  Taking the trace of the solution on $t=0$ then
yields non-trivial initial data which generates this nontrivial
interacting solution.  By finite speed of propagation, for any fixed
$x$, the solution will be stationary for $t$ large enough, so the
solution is eventually noninteracting.

Because of Lemma~\ref{pcni}(b), \emph{all} non-compressive shock-free
solutions with piecewise constant monotone entropy profile can be
generated in this way, and we expect that all such non-compressive
shock-free solutions have this structure as long as the entropy
profile is monotonic.

\paragraph{Rarefactions with two monotonic contacts}

By way of example, we explicitly construct an eventually
noninteracting solution consisting of rarefactions with two increasing
entropy jumps ($3$-contacts).  For convenience we set $m_{-\infty}=1$
and define the entropy profile by jumps $Q_0>1$ and $Q_1>1$ at $x_0$
and $x_1$, respectively, see~\eq{q def}.

We specify Cauchy data by $r(t,x_0-)$ and $s(t,x_0-)=0$, with 
\[
  \dot r(t,x_0-):=r_t(t,x_0-)>0, \com{supported on} t\in[0,T].
\]
This means that there is a simple backward rarefaction wave emerging
from $x_0-$, and no incoming forward wave.  Equivalently, by \eq{rs
  def}, we choose
\[
  z(t,x_0-) = Z(t), \quad u(t,x_0-) = - Z(t) \com{for} t\in[0,T],
\]
where $Z(t)$ is a positive monotone decreasing function.

By applying  Corollary~\ref{TYcor} at both jumps, it follows that all
waves in the solution are rarefactions, so that $\dot s\leq 0$ and
$\dot r\geq0$ everywhere.  In particular, the maximum and minimum
values of $z$ are taken on at $(x_0-,0)$ and $(x_1+,\infty)$,
respectively.  Since the solution is eventually noninteracting, by
\eq{q def}, it follows that we have the uniform global bounds
\beq
   Z_* \le z(x,t) \le Z^*,
\label{zbd}
\eeq
where the bounds are given by
\beq
   Z^* \ge Z(0) \com{and}
   Z_* \le Q_0^{\frac{-2}{d+1}}\,Q_1^{\frac{-2}{d+1}}\,Z(T).
\label{zstar}
\eeq

We recall Lax's estimate of gradient growth in the
$p$-system~\cite{lax2,lax3}.
\begin{lemma}
In a constant entropy field, the gradients of Riemann invariants of
$C^2$ solutions satisfy
\begin{align}
  \frac1{-\dot s(B)}&=\frac1{-\dot s(A)}\,R(A,B)
     + \int K(x,B)\,dx,\nn\\
  \frac1{\dot r(B)}&=\frac1{\dot r(A)}\,R(A,B)
	- \int K(x, B)\,dx,
\label{lax lemma}
\end{align}
where the integrations are along the forward and backward
characteristics connecting the points $A=(x_A,t_A)$ and $B=(x_B,t_B)$,
respectively, and where
\[
  R(A,B) := \left(\frac{z(A)}{z(B)}\right)^{d/2} \com{and}
  K(A,B) :=\frac{d\,R(A,B)}{2\,m^2\,z^{d+1}(A)}.
\]
\end{lemma}

For convenience we choose $x_1$ large enough (or $T$ small enough)
that the backward wave to the right of $x_0$ meets the entropy jump at
$x_1$ in negative time $t<0$, so that only the forward wave crosses
the second jump in our region of interest.  This is clearly possible
because we have uniform bounds \eq{zbd} for $z$, and hence for the
wavespeed.  For the same reason, the backward wave to the right of
$x_1$ meets the $t$-axis in some bounded interval $[x_1,X^*]$, so that
the support of the data for the corresponds initial value problem is
$[x_0,X^*]$.  By the mean value theorem, integrating the
characteristics from $(x_0+,T)$, we get
\begin{align}
  x_1 - x_0 &\ge \ol c_1\,T \nn\\
  x_1 - x_0 &= \ol c_2\,(T_*-T), \com{and}\label{xct}\\
  X_* - x_1 &= \ol c_3\,T_*, \nn
 \end{align}
for some values $\ol c$, where $T_*$ is the time at which the forward
characteristic from $(x_0+,T)$ meets $x_1$.  In fact, since the
forward wave is simple, $\ol c_2$ is exactly $c(z(T,x_0+))$.

It remains to show that the backward rarefactions focus in the
half-plane $t<0$.  It is convenient to describe the Riemann
invariants in terms of the data $Z(t)$, as follows.  First, note that
\[
   \dot r(x_0-,t)=-\dot s(x_0-,t)=-2\,\dot Z(t) \ge 0,
\]
so that, by \eq{rsjump},
\begin{align}
  \dot r(x_0+,t) &= -(1+Q_0)\,\dot Z(t) \com{and}\nn\\
  -\dot s(x_0+,t) &= -(Q_0-1)\,\dot Z(t).
\label{rsinit}
\end{align}

By construction, there are two cases to consider, namely, a backward
characteristic from $x_0+$, and a forward characteristic from $x_0+$
to $x_1-$ followed by a backward characteristic from $x_1+$.  For the
first case, we apply \eq{lax lemma} with $A=(x_0+,t)$, and $B=(x,0)$,
with $t\le T$ and $x_0\le x\le x_1$.  We require
\[
  \frac1{\dot r(B)}=\frac1{\dot r(A)}\,R(A,B)
	- \int K(x, B)\,dx > 0,
\]
which reduces to 
\[
  \dot r(A) < \frac{R(A,B)}{\int K(x, B)\,dx},
\]
and which clearly follows if
\beq
  -\dot Z(t) < \frac1{(x_1-x_0)\,(1+Q_0)}\,\frac{\min R}{\max K}.
\label{z1}
\eeq
Now consider a forward characteristic from $A$ to $C=(x_1,t_1)$,
followed by a backward characteristic from $C$ to $D=(x,0)$.  Since
$-\dot s\ge0$, \eq{lax lemma} implies that
\[
  -\dot s(C-) \le -\dot s(A)\,\frac1{R(A,C-)},
\]
and again applying \eq{rsjump}, we get
\begin{align*}
  \dot r(C+) &= -\dot s(C-)\,\frac{Q_1-1}2 
   \le -\dot s(A+)\,\frac{Q_1-1}{2\,R(A,C-)} \\
   &\le -\dot Z(t)\,\frac{(Q_1-1)\,(Q_0-1)}{2\,R(A,C-)}.
\end{align*}
As in the first case, the backward rarefaction focusses at $t<0$
provided
\[
  \dot r(C+) < \frac{R(C+,D)}{\int K(x, D)\,dx},
\]
which certainly holds if
\beq
  -\dot Z(t) < \frac2{(Q_1-1)\,(Q_0-1)}\,
		\frac{(\min R)^2}{(X_*-x_1)\,\max K}.
\label{z2}
\eeq

It follows that if \eq{z1} and \eq{z2} hold, then our solution
satisfies the required properties.  We now further estimate \eq{z1}
and \eq{z2}.  We fix the bounds $Z_*$ and $Z^*\ge Z(0)$ for the
data $Z(t)$, use \eq{xct} to eliminate $X_*-x_1$, and simplify, to
find constants  $K_i$ depending only on $Z_*$, $Z^*$ and $Q_i$, such
that, if
\beq
  -\dot Z(t) < \min\left\{\frac {K_1}{x_1-x_0},\ 
                          \frac{K_2}{x_1-x_0+\ol c_2\,T}\right\},
\label{zeq}
\eeq
then \eq{z1} and \eq{z2} hold.

Finally, we show consistency of the construction, as follows.  First,
fix $Z_*$, $Z^*$ and $Q_i$, and choose $x_0$, $x_1$ and $T$ such that
\[
  x_1-x_0 \ge C_1 \,T,
\]
where $C_1$ is an upper bound for the wave speed in the region
$(x_0,x_1)$.  Now choose $Z(t)$ such that \eq{zeq} holds \emph{and}
such that
\[
  Z^* \ge Z(0) \ge Z(T) \ge Q_0^{\frac2{d+1}}\,Q_1^{\frac2{d+1}}\,Z_*.
\]
It is evident that these conditions are consistent by further
requiring, say,
\[
  0 \le -\dot Z(t) \le 
    \frac{Z^*-Q_0^{\frac2{d+1}}\,Q_1^{\frac2{d+1}}\,Z_*}T.
\]

\section{Non-monotonic Contact Discontinuities}

We have shown that if the entropy profile is monotonic, then
shock-free solutions are eventually noninteracting or contain an
asymptotic vacuum.  Here, by example, we show that this need not be
true when the entropy profile is non-monotonic.  Our entropy profile
is piecewise constant with two contacts.  As in earlier sections, we
characterize the jumps using \eq{q def}.  Thus, we place a $3$-contact
$Q_0>1$ at $x_0$ and a $1$-contact $Q_1<1$ at $x_1$; without loss of
generality we assume $Q_0=Q=1/Q_1$.  We assume also that there are no
incoming waves on either side of the interaction region, so the
interactions are confined to the strip $x_0<x<x_1$.

We consider the initial value problem with $C^2$ data prescribed in
the interval $(x_0,x_1)$, and constant outside that interval.  We note
that if the data is anywhere compressive, shocks must form in finite
time, as the compressions cannot be cancelled at the entropy jumps.

\begin{lemma}
Fix the entropy profile as described above, and assume the initial
data is constant outside the interval $[x_0,x_1]$.  Then the solution
is globally shock-free if and only if the data is nowhere compressive.
\end{lemma}

\begin{proof}
It suffices to show that rarefactive data produces global shock-free
solutions.  We first solve the initial boundary value problem in the
region
\[
  \Omega = \{(x,t)\,|\,x_0<x<x_1,\ t>0\},
\]
with boundary data prescribed by the requirement that there are no
incoming waves.  We then resolve the states across the entropy jumps
and propagate the solution outwards as simple (rarefaction) waves.

We obtain the boundary conditions by setting
\[
  s(x_0-,t)=s(x_0-,0) \com{and} r(x_1+,t)= r(x_1+,0),
\]
and solving the Hugoniot conditions \eq{rsjump}.  After
simplification, we write the boundary conditions as
\begin{align}
  s(x_0+,t) &= \frac{1-Q}{1+Q}\,r(x_0+,t)+\frac{2\,Q}{1+Q} \,s(x_0-,0), \nn\\
  r(x_1-,t) &= \frac{1-Q}{1+Q}\,s(x_1-,t)+\frac{2\,Q}{1+Q} \,r(x_1+,0).
\label{rsbc}
\end{align}

According to Corollary~\ref{TYcor}, the waves that are reflected back
into the domain are always rarefactive, so we obtain global bounds on
the derivatives of Riemann invariants.

We now apply the existence theorem of \cite{li1}, Chap.~6, Thm.~3.1,
which states that there is a unique global $C^1$ solution in $\Omega$
which is nowhere compressive, provided \emph{a priori} bounds are
satisfied.  We obtain explicit time-dependent bounds below, which
suffice for application of the theorem.

Finally, having solved the IBVP inside the domain $\Omega$, we again
apply \eq{rsjump} to obtain the Riemann invariants outside the domain,
which are rarefactions by construction.  These thus propagate for all
times without forming shocks.
\end{proof}

It is clear that these solutions are \emph{not} eventually
noninteracting, as there is always reflected rarefaction.  It follows
that either the solution forms an asymptotic vacuum, or waves continue
to reflect back and forth between $x_0$ and $x_1$ for all times.

\begin{figure}[htp] \centering
\begin{tikzpicture}[scale=0.4]
\draw[->] (-3.5,0) -- (-3.5,2);
\node at (-3.7,1.8) {$t$};
\draw[->] (-3.5,0) -- (-1.5,0);
\node at (-1.7,-0.3) {$x$};
\foreach \xx in {1,4}
  \draw[dashed] (\xx,0) -- (\xx,5.45);
\foreach \y/\s in {0/1,3.2/1.5}
  \draw (4,\y) .. controls (3,\y+\s*0.3) and (2,\y+\s*0.5) .. (1,\y+\s*1.5);
\draw (1,1.5) .. controls (2,1.9) and (3,2.5) .. (4,3.2);
\node at (1,-0.3) {$x_0$};
\node at (4,-0.3) {$x_1$};
\node at (0.6,1.5) {$t_1$};
\node at (4.4,3.2) {$t_2$};
\node at (0.6,5.4) {$t_3$};
\end{tikzpicture}
  \caption{Infinite reflection of characteristics}
\label{nonmonoexample}
\end{figure}
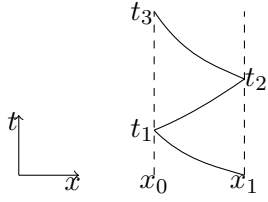

We analyze a general rarefactive solution, as follows.  Starting from
the corner $(x_1,0)$, we trace the reflected characteristic through
the solution.  Let $t_k$ denote the time of the $k$-th intersection of
this characteristic with the boundary, as in
Figure~\ref{nonmonoexample}, and use the subscript to denote the
corresponding interior state, so $z_{2k} = z(x_1-,t_{2k})$ and
$z_{2k+1} = z(x_0+,t_{2k+1})$.  We use \eq{simple} to describe the
states at either end of the backward and forward characteristics,
respectively, by
\begin{align*}
  u_{2k} - u_{2k+1} &= m\,(z_{2k}-z_{2k+1}),\\
  u_{2k} - u_{2k-1} &= m\,(z_{2k-1}-z_{2k}).
\end{align*}
Since the waves outside $\Omega$ are simple, we obtain
\begin{align*}
  u_{2k+2} - u_{2k} &= \frac mQ\,(z_{2k+2}-z_{2k}),\\
  u_{2k+1} - u_{2k-1} &= \frac mQ\,(z_{2k-1}-z_{2k+1}),
\end{align*}
where we have used \eq{z_q_rel} to express the outside states in terms
of the interior states.  Eliminating $u_n$, we get the same equation
for even and odd $n$,
\beq
  (1+Q)\,z_{n+1} = 2\,Q\,z_{n} + (1-Q)\,z_{n-1},
\label{zQ}
\eeq
a linear difference equation for $z_n$.  We obtain a linear equation
because the simple wave description \eq{simple} and the jump
conditions \eq{z_q_rel} are linear in $u$ and $z$, for $m$ constant.

Setting $z_n=\lambda^n$, we get
\[
  \lambda^2 - (1+\eta)\,\lambda + \eta 
  = (\lambda-1)\,(\lambda-\eta) = 0,
\]
where we have set
\[
  \eta := \frac{Q-1}{Q+1} \in (0,1).
\]
It follows that the general solution of \eq{zQ} is
\beq
  z_n = \frac{z_1-\eta\,z_0}{1-\eta} + \eta^n\,\frac{z_0-z_1}{1-\eta}.
\label{zsoln}
\eeq

\begin{thm}
For entropy profile as given above, and rarefactive initial data
prescribed on $(x_0,x_1)$, there are three possible long-time
behaviors: asymptotic vacuum; infinitely reflected waves which
converge to the vacuum state; and infinitely reflected waves with
non-vanishing wavespeed and density.  The long-time behavior is
determined by the expression $\zeta = z_1-\eta\,z_0$, with a
bifurcation at $\zeta=0$.
\end{thm}

\begin{proof}
Since the data is rarefactive, $z_1<z_0$, so the second term of
\eq{zsoln} is positive and decreasing with $n$.  If $\zeta>0$, then
$z_n$ is defined for all $n$ and approaches $\zeta/(1-\eta)>0$ as
$n\to\infty$.  This implies that the waves are reflected infinitely
often with uniformly bounded wavespeed.

If $\zeta = 0$, then $z_n=\eta^n\,z_0$, which clearly converges to the
vacuum as $n\to\infty$.  Moreover, since $z_n$ is defined for all $n$,
the waves interact with the entropy jumps infinitely often, and the
$n$-th characteristic has nonzero speed, so the solution does
\emph{not} contain an asymptotic vacuum.

Finally, if $\zeta<0$, then for some $N$, \eq{zsoln} yields $z_N<0$,
which contradicts the physical requirement that $z>0$.  We conclude
that $z_N$ cannot be defined, which means that the $N$-th
characteristic never meets the boundary, and we therefore have an
asymptotic vacuum.
\end{proof}

\begin{cor}
Solutions with nonmonotic entropy may converge to vacuum as
$t\to\infty$ even though they contain no asymptotic vacuum.  In
particular, the condition of an asymptotic vacuum is stronger than the
vacuum condition \eq{vacuum2}.
\end{cor}

\begin{proof}
We have seen that $\zeta=0$ yields a vacuum in the limit as
$t\to\infty$ for all $x_0\le x\le x_1$.  We show that in case
$\zeta=0$, \eq{vacuum2} holds as an equality.  By \eq{z_q_rel}, at
$x_1$ and $x_0$, respectively, we have
\[
  u_0(\infty) = u(x_1-,0+), \quad
  m(\infty)\,z_0(\infty) = \frac mQ\,z_0
\]
and
\[
  u(x_0-,t_1) = u(x_0+,t_1),\quad
  m\,z_1 = Q\,m(-\infty)\,z(x_0-,t_1),
\]
where $z_0=z(x_1-,0+)$ and $z_1=z(x_0+,t_1)$.  Now, by \eq{simple} we
also have
\begin{gather*}
  u(x_0-,t_1)-u_0(-\infty) =
  m(-\infty)\,(z_0(-\infty)-z(x_0-,t_1)),\\
  u(x_1-,0+)-u(x_0+,t_1) = m\,(z_0-z_1).
\end{gather*}
Now, $m(-\infty)=m(\infty)$, and we calculate
\begin{align*}
  u_0(\infty)-u_0(-\infty) &=
    m(\infty)\,(z_0(-\infty)- z_0(x_0-,t_1)) + m\,(z_0-z_1)\\
   &= m(\infty)\,z_0(-\infty) 
    -\frac mQ\,z_1-m\,z_1+Q\,m(\infty)\,z_0(\infty),
\end{align*}
which yields
\begin{align*}
  u_0(\infty)-u_0(-\infty)
    &-m(-\infty)\,z_0(-\infty)-m(\infty)\,z_0(\infty)\\
  &= -\frac mQ\,z_1-m\,z_1+(Q-1)\,\frac mQ\,z_0.
\end{align*}
This vanishes since $z_1 = \eta\,z_0 = \frac{Q-1}{Q+1}\,z_0$.
\end{proof}

\begin{cor}
At the bifurcation point $\zeta=0$, the solution asymptotically
approaches vacuum at the rate
\[
   z_0\,\left(1+\frac{t_n\,c_0}{x_1-x_0}\,(\eta^{-d}-1)\right)^{-1/d}
   \le z_n \le
   z_0\,\left(1+\frac{t_n\,c_0}{x_1-x_0}\,(1-\eta^d)\right)^{-1/d},
\]
so, in particular, $z(t) \sim O(1)\,(1+t)^{-1/d}$, which implies
\[
  \tau = O(1)\,(1+t)^{1-1/d}, \com{or} \rho = O(1)\,(1+t)^{1/d-1}.
\]
\end{cor}

\begin{proof}
Since $z$ is monotonic along characteristics, we estimate the
interaction times $t_n$ by
\[
  \frac{x_1-x_0}{c(z_n)} \le t_{n+1} - t_n \le
  \frac{x_1-x_0}{c(z_{n+1})},
\]
so that
\[
  \frac{x_1-x_0}m\,\sum_{k=0}^{n-1} z_k^{-d} \le t_n
    \le \frac{x_1-x_0}m\,\sum_{k=1}^{n} z_k^{-d}.
\]
If $\zeta>0$, then the terms $z_k^{-d}$ converge to
$\kappa=(1-\eta)^d/\zeta^d\ne0$, so the times do not converge and
$t_n\sim \kappa\,n$ for $n$ large.

When $\zeta = 0$, we get
\[
  \frac{x_1-x_0}{m\,z_0^d}\,\sum_{k=0}^{n-1} \eta^{-kd} \le t_n
    \le \frac{x_1-x_0}{m\,z_0^d}\,\sum_{k=1}^{n} \eta^{-kd},
\]
which simplifies to
\[
  \frac{\eta^{-nd}-1}{\eta^{-d}-1} \le \frac{t_n\,m\,z_0^d}{x_1-x_0}
     \le \eta^{-d}\,\frac{\eta^{-nd}-1}{\eta^{-d}-1},
\]
which in turn yields
\[
  1+\frac{t_n\,c_0}{x_1-x_0}\,(1-\eta^d) \le \eta^{-nd}
  \le 1+\frac{t_n\,c_0}{x_1-x_0}\,(\eta^{-d}-1),
\]
where $c_0=c(z_0)=m\,z_0^d$.  Since $z_n = \eta^n\,z_0$, the corollary
follows.
\end{proof}

\section{Solutions with One Shock}

We briefly analyze the evolution of a single shock, as follows.  The
shock curve and states on either side satisfy the Rankine-Hugoniot
relations \eq{RH}.  We analyze the interacting shock by treating it as
a free boundary problem and imposing the conditions that the flow on
either side of the shock is either isentropic or stationary.  This
yields a nontrivial interacting shock which is however not difficult
to analyze.

We parameterize the shock curve, $\Sigma = (x(a),t(a))$, together with
Cauchy data on either side of $\Sigma$, subject to the
Rankine-Hugoniot conditions.  According to \eq{PWmf} and \eq{PWgh},
we have
\begin{align}
    \frac{z_1}{z_0} &= a, \nn\\
    \frac{m_1}{m_0} &= f(a), \nn\\
    \frac{p_1}{p_0} &= a^{d+1}\,f(a)^2, \nn\\
    \xi &= \pm \ol m\,\ol z^d\,g(a), \label{shock}\\
    u_1 - u_0 &= \pm \ol m\,\ol z\,h(a), \nn\\\nn
    t(a) - t_0 &= \pm \int_0^a 
      \frac{\dot x(\tilde a)}{\ol m\,\ol z^d\,g(\tilde a)}\;d\tilde a,
\end{align}
where
\[
   \ol m = m_0\,\frac{1+f(a)}2, \quad 
   \ol z = \left[\frac{(d+1)\,p_1}{m_0^2}\right]^{1/(d+1)}
\]
are the averages, and subscripts refer to opposite sides of the shock;
here the parameter satisfies $1<a<d^{1/(d-1)}$.

Once we specify consistent Cauchy data, we can extend the solution to
a neighborhood around the shock by solving the Cauchy problem locally
on either side of the shock.  Because we implicitly assume the Lax
entropy conditions, the Cauchy problem is non-characteristic as long
as the shock has nonzero strength, and so general existence theorems
apply~\cite{liyu,li1,Dafermos}.  Moreover, if we establish \emph{a
  priori} estimates for gradients, global existence follows.

We specify the Cauchy data by assuming that the flow on one
side of the shock is stationary, while the flow on the other side is
isentropic.  We do this by specifying that
\beq
  u_1(a) = U_1,\quad p_1(a) = P_1, \com{and} m_0(a) = M_0,
\label{CP}
\eeq
so that $1$ refers to the stationary solution and $0$ to the
isentropic solution.  This leaves a single free function, namely
$x(a)$, all other quantities being determined by \eq{shock}, \eq{CP}.

\begin{lemma}
There are globally defined interacting solutions containing one shock
which separates the plane into two regions, such that the solution is
isentropic in one region and stationary in the other.
\end{lemma}

\begin{proof}
Local existence follows because the Cauchy problem is not
characteristic.  Assume the shock is backward, with stationary
solution behind the shock, so $u_1=u_r$ is constant, and the sign
choices in \eq{shock} are all $-$.  It is clear that the stationary
solution behind the shock is globally defined.  To verify existence in
the isentropic region, we must estimate the derivatives of the Riemann
invariants.

To calculate the Riemann invariants, we differentiate the invariant
along $\Sigma$, to get
\beq
  \frac{dr}{da} = r_x\,\frac{dx}{da} + r_t\,\frac{dt}{da}
     = r_x\,\frac{dx}{da}\,\left(1-\frac{c(a)}{\xi(a)}\right),
\label{drda}
\eeq
where we have used $r_t-c\,r_x=0$ and $\frac{dx}{dt}=-\xi$, and
similarly
\beq
  \frac{ds}{da} = s_x\,\frac{dx}{da} + s_t\,\frac{dt}{da}
     = s_x\,\frac{dx}{da}\,\left(1+\frac{c(a)}{\xi(a)}\right).
\label{dsda}
\eeq
We now calculate from \eq{shock} and \eq{taupc} that
\begin{gather*}
  m_1(a) = M_0\,f(a), \quad
  P_1 = \frac{M_0^2}{d+1}\,f(a)^2\,(a\,z_0(a))^{d+1},
  \com{and}\\
  U_1 - u_0(a) = 
     - M_0\,\left[\frac{(d+1)\,P_1}{M_0^2}\right]^{1/(d+1)}
     \,\frac{1+f(a)}2\,h(a).
\end{gather*}
We use these to solve for $z_0(a)$ and $u_0(a)$, and plug in to
\eq{drda} and \eq{dsda} to get
\beq
  r_x = \frac{da}{dx}\,\mathcal R(a) \com{and}  
  s_x = \frac{da}{dx}\,\mathcal S(a),
\label{rxsx}
\eeq
for explicit functions $\mathcal R$ and $\mathcal S$.

We now restrict $a$ to a compact subinterval of $(1,d^{1/(d-1)})$,
which implies uniform bounds on all thermodynamic variables, and so
also $\mathcal R(a)$ and $\mathcal S(a)$, and choose $x(a)$ so that
$r_x$ and $s_x$ are small enough on $\Sigma$ that \eq{lax lemma}
yields finite bounds for $t\ge 0$.

The other case, of stationary solutions before the shock, is similar.
Proceeding as above, we obtain \eq{rxsx} (with slightly different
$\mathcal R$ and $\mathcal S$), and we now must further restrict our
data to ensure that $s_x \ge 0$, so the outgoing forward wave is a
rarefaction, and we must ensure that the backward wave does not focus
in the halfplane $t>0$.
\end{proof}

It is interesting to ask whether the shock persists.  Disappearance of
the backward shock occurs in the limit $a\to1^+$, so we must
demonstrate that this limit cannot occur for finite $(x(a),t(a))$.  It
is well known that in the limit of vanishing wave strength, we have
\[
  1 - \frac{c(a)}{\xi(a)} = O(a-1) \com{as} a\to 1^+,
\]
see~\cite{smoller}.  The functions $\mathcal R(a)$ and $\mathcal S(a)$
remain bounded away from $0$ in the limit $a\to1$, so we obtain
\[
  |r_x| = \frac{O(1)}{a-1}\,\frac{da}{dx},
\]
and in particular, if $r_x$ remains finite, by \eq{drda}, there exist
$\nu>0$ and $\delta>0$ such that
\[
  (a-1)\,\frac{dx}{da} > \nu \com{for} 1<a<1+\delta.
\]
Now, for $0<\epsilon<\delta$, we have
\begin{align*}
  x(1+\delta)-x(1+\epsilon) &= 
  \int_{1+\epsilon}^{1+\delta}\frac{dx}{d\tilde a}\;d\tilde a\\
  &\ge \int_{1+\epsilon}^{1+\delta}\frac{\nu}{\tilde a-1}\;d\tilde a
  \\&= \nu\,\log\frac\delta\epsilon \to \infty,
\end{align*}
as $\epsilon\to0$, so $x\to -\infty$ as $a\to1^+$.  This implies that
the shock strength can vanish only at infinity.

\end{document}